\def\version{[arXiv 1: 09/05/2013]}
\documentclass[11pt,a4paper]{amsart}
\usepackage{amssymb,amscd,amsxtra,eucal}
\usepackage[all]{xy}
\usepackage{fullpage}

\theoremstyle{plain}

\newtheorem{thm}{Theorem}[section]
\newtheorem{lem}[thm]{Lemma}
\newtheorem{prop}[thm]{Proposition}
\newtheorem{cor}[thm]{Corollary}

\theoremstyle{definition}

\numberwithin{equation}{section}

\def\ie{\emph{i.e.}}

\def\ds{\displaystyle}
\def\:{\colon}
\def\.{\cdot}

\def\<{\left\langle}
\def\>{\right\rangle}
\def\({\left(}
\def\){\right)}
\def\ph#1{\phantom{#1}}
\def\epsilon{\varepsilon}
\def\phi{\varphi}
\def\subset{\subseteq}

\def\leq{\leqslant}
\def\geq{\geqslant}
\def\la{\leftarrow}
\def\ra{\rightarrow}
\def\lla{\longleftarrow}
\def\lra{\longrightarrow}
\def\Lra{\Longrightarrow}
\def\bar#1{\overline{#1}}

\def\tilde#1{\widetilde{#1}}
\def\iso{\cong}
\DeclareMathOperator{\DERIV}{d}\renewcommand{\d}{\DERIV}

\DeclareMathOperator{\im}{im}

\DeclareMathOperator{\coker}{coker}

\def\E{\mathrm{E}}

\def\k{\Bbbk}

\def\Z{\mathbb{Z}}

\def\oTimes#1{{\ds\mathop{\otimes}_{#1}}}

\def\ideal{\triangleleft}

\DeclareMathOperator{\Ext}{Ext}

\DeclareMathOperator{\Hom}{Hom}

\DeclareMathOperator{\Tor}{Tor}

\def\del{\partial}
\def\dh{\d^{\mathrm{h}}}
\def\dv{\d^{\mathrm{v}}}
\DeclareMathOperator{\TOT}{Tot}
\def\Tot{\TOT^\oplus}
\def\dt{\d^{\TOT}}
\DeclareMathOperator{\bideg}{bideg}

\begin{document}
\title[On the homology of regular quotients]
{On the homology of regular quotients}
\author{Andrew Baker}
\address{School of Mathematics \& Statistics,
University of Glasgow, Glasgow G12 8QW, Scotland.}
\email{a.baker@maths.gla.ac.uk}
\urladdr{http://www.maths.gla.ac.uk/$\sim$ajb}
\thanks{\hfill\version \\
I would like to thank Alain Jeanneret, Andrey Lazarev
and Samuel W\"uthrich for helpful comments, also
Peter May and Larry Smith for pointing out the related
papers~\cite{Smith:EM,Tate}, and finally the Glasgow
Derived Categories Seminar which provided an opportunity
for the author to learn some useful algebra.
}
\begin{abstract}
We construct a free resolution of $R/I^s$ over $R$ where $I\ideal R$ is
generated by a (finite or infinite) regular sequence. This generalizes
the Koszul complex for the case $s=1$. For $s>1$, we easily deduce that
the algebra structure of $\Tor^R_*(R/I,R/I^s)$ is trivial and the reduction
map $R/I^s\lra R/I^{s-1}$ induces the trivial map of algebras.
\end{abstract}
\maketitle
\section*{Introduction}

Let $R$ be a commutative unital ring. We will say that an ideal
$I\ideal R$ is \emph{regular} if it is generated by a regular
sequence $u_1,u_2,\ldots$ which may be finite or infinite. We will
call the quotient ring $R/I$ a \emph{regular quotient} of $R$. All
tensor products and homomorphisms will be taken over $R$ unless
otherwise indicated.

It is well known, see~\cite{Weibel} for example, that there is a
\emph{Koszul resolution}
\[
\mathbf{K}_*\lra R/I\ra0,
\]
where
\[
\mathbf{K}_*=\Lambda_R(e_i:i\geq1)
\]
is a differential graded algebra with $e_i$ in degree~$1$ 
and differential given by $\d e_i=u_i$. The following 
result is standard, see for example~\cite{Matsumura,Weibel}.
\begin{prop}\label{prop:KoszulRes}
If $I\ideal R$ is regular, then $\mathbf{K}_*$ provides 
a free resolution of $R/I$ over $R$. Moreover, 
$(\mathbf{K}_*,\d)$ is a differential graded $R$-algebra.
\end{prop}
\begin{cor}\label{cor:KoszulRes}
As $R/I$-algebras,
\[
\Tor^{R}_*(R/I,R/I)=\Lambda_{R/I}(e_i:i\geq1).
\]
\end{cor}
We will generalize this by defining a family of free resolutions
\[
\mathbf{K}(R;I^s)_*\lra R/I^s\ra0\quad(s\geq1),
\]
which are well related and allow efficient calculation of
the $R/I$-algebra $\Tor^R_*(R/I,R/I^s)$.

The resolution we construct may well be known, however lacking a
convenient reference we give the details. Our immediate motivation
lies in topological calculations that are part of joint work with
A.~Jeanneret and A.~Lazarev~\cite{AB+AJ:Brave-BocksteinOps,AB+AL:Rmod-ASS},
but we believe this algebraic construction may be of wider interest.
Our approach to this construction was suggested by derived category
ideas and in particular the construction of Cartan-Eilenberg
resolutions~\cite{CE,Weibel}. Tate's method of killing homology
classes~\cite{Tate} seems to be related, as does Smith's work on
homological algebra~\cite{Smith:EM}, but neither appears to give
our result explicitly.

\subsection*{History}
This paper was originally written at the beginning of the Millennium
and appeared as 
\emph{Glasgow University Mathematics Department preprint no.~01/1}
but it was never formally published. Subsequently, Samuel W\"uthrich
developed the algebra further in~\cite{SW:HomPowRegId} and applied 
it in studying $I$-adic towers. However, our original approach 
still seems of interest and we feel it worthwhile making it more 
easily available. 

\subsection*{Notation}
Our indexing conventions are predominantly homological (\ie, lower
index) as opposed to cohomological, since that is appropriate for
the topological applications we have in mind. Consequently, complexes
have differentials which \emph{decrease} degrees.

For a complex $(C_*,d)$, we define its \emph{$k$-fold suspension}
$(C[-k]_*,d[-k])$ by
\[
C[-k]_n=C_{n-k},\quad d[-k]=(-1)^{k}d\:C_{n-k}\lra C_{n-k-1}.
\]
For an $R$-module $M$, we sometimes view $M$ as the complex with
\[
M_n=
\begin{cases}
M&\text{if $n=0$}, \\
0& \text{if $n\neq0$}.
\end{cases}
\]

\section{A resolution for $R/I^s$}\label{sec:Res-R/Is}

In this section we describe an explicit $R$-free resolution for $R/I^s$
which allows homological calculations. We begin with a standard result;
actually the cited proof applies when $I$ is finitely generated, but the
adaption to the general case is straightforward. We will always interpret
$I^0/I$ as $R/I$.
\begin{lem}[\cite{Matsumura}, Theorem 16.2]\label{lem:Is/I(s+1)-R/I-free}
For $s\geq0$, $I^s/I^{s+1}$ is a free $R/I$-module with a basis consisting
of the residue classes of the distinct monomials of degree $s$ in the $u_i$.
\end{lem}
\begin{cor}\label{cor:Is/I(s+1)-R/I-free}
For $s\geq0$, there is a free resolution of $I^s/I^{s+1}$ over $R$ of
the form
\[
\mathbf{Q}^{(s)}_*=
\mathbf{K}_*\otimes\mathrm{U}^{(s)}\lra I^s/I^{s+1}\ra0,
\]
where $\mathrm{U}^{(s)}$ is a free $R$-module on a basis indexed by
the distinct monomials of degree $s$ in the generators $u_i$.
\end{cor}
For a sequence $\mathbf{i}=(i_1,\ldots,i_s)$ and its associated monomial
$u_{\mathbf{i}}=u_{i_1}\cdots u_{i_s}$, we will denote the corresponding
basis element $1\otimes u_{i_1}\cdots u_{i_s}$ of
$\mathbf{K}_*\otimes\mathrm{U}^{(s)}$ by $\tilde u_{\mathbf{i}}$ and more
generally $x\otimes\tilde u_{\mathbf{i}}$ by $x\tilde u_{\mathbf{i}}$. We
will also denote the differential on $\mathbf{Q}^{(s)}_*$ by
$\d_{\mathbf{Q}}^{(s)}$, noting that
\begin{equation}\label{eqn:d(s)-Qs}
\d_{\mathbf{Q}}^{(s)}x\tilde{u_{\mathbf{i}}}=(\d x)\tilde{u_{\mathbf{i}}}.
\end{equation}
For $s\geq0$, there is also a map
\[
\del^{(s+1)}\:\mathbf{Q}^{(s)}_*\lra\mathbf{Q}^{(s+1)}_{*-1};
\quad
\del^{(s+1)}\sum_{\mathbf{i}}y_{\mathbf{i}}\tilde u_{\mathbf{i}}
=\sum_{\mathbf{i}}(\d y_{\mathbf{i}})\tilde u_{\mathbf{i}},
\]
where we interpret the products for $y_{(i_1,\ldots,i_s)}\in\mathbf{K}_*$
according to the formula
\begin{align*}
(\d y_{(i_1,\ldots,i_s)})\tilde u_{(i_1,\ldots,i_s)}&=
\sum_{j}y_{(i_1,\ldots,i_s),j}\;\tilde u_{(i_1,\ldots,i_s,j)} \\
\intertext{with}
\d y_{(i_1,\ldots,i_s)}&=
\sum_{(i_1,\ldots,i_s),j}y_{(i_1,\ldots,i_s),j}\tilde u_j.
\end{align*}

For $s\geq1$, define
\[
\mathbf{K}(R;I^s)_*=
\mathbf{Q}^{(0)}_*\oplus\mathbf{Q}^{(1)}_*\oplus
\cdots\oplus\mathbf{Q}^{(s-1)}_*,
\]
with the differential $\d^{(s)}$ given by
\begin{equation}\label{eqn:d(s)-Defn}
\d^{(s)}(x_0,x_1,\ldots,x_{s-1})=(x'_0,x'_1,\ldots,x'_{s-1}),
\end{equation}
where
\[
x'_k=
\begin{cases}
\d_{\mathbf{Q}}^{(0)}x_0& \text{if $k=0$}, \\
\del^{(k)}x_{k-1}+\d_{\mathbf{Q}}^{(k)}x_k& \text{otherwise}.
\end{cases}
\]
We need to show that $(\d^{(s)})^2=0$. This follows from the following
easily verified identities which hold for all $r\geq0$:
\begin{align}
\d_{\mathbf{Q}}^{(r+1)}\del^{(r+1)}+\del^{(r+1)}\d_{\mathbf{Q}}^{(r)}&=0,
\label{eqn:del-dQ}\\
\del^{(r+1)}\del^{(r)}&=0.
\label{eqn:del^2}
\end{align}
Then
\[
(\d^{(s)})^2(x_0,x_1,\ldots,x_{s-1})=(x''_0,x''_1,\ldots,x''_{s-1}),
\]
where
\begin{align*}
x''_0&=(\d^{(0)})^2x_0=0, \\
x''_1&=
\del^{(1)}\d_{\mathbf{Q}}^{(0)}x_0+\d^{(1)}\del^{(1)}x_0+(\d^{(1)})^2x_1=0, \\
\intertext{while for $2\leq k\leq s-1$,}
x''_k&=
\del^{(k)}\del^{(k-1)}x_{k-2}
+\del^{(k)}\d_{\mathbf{Q}}^{(k-1)}x_{k-1}+
\d_{\mathbf{Q}}^{(k)}\del^{(k)}x_{k-1}+(\d_{\mathbf{Q}}^{(k)})^2x_{k})=0.
\end{align*}

There is an augmentation map
\[
\epsilon^{(s)}\:\mathbf{K}(R;I^s)_0\lra R/I^s,
\]
namely the $R$-module homomorphism
\begin{multline*}
\epsilon^{(s)}
\(a_0,\sum_{(i_1)}a_{(i_1)}\tilde u_{(i_1)},
\sum_{(i_1,i_2)}a_{(i_1,i_2)}\tilde u_{(i_1,i_2)},\ldots,
\sum_{(i_1,i_2,\ldots,i_{s-1})}
a_{(i_1,i_2,\ldots,i_{s-1})}\tilde u_{(i_1,i_2,\ldots,i_{s-1})}\) \\
=
a_0+\sum_{(i_1)}a_{(i_1)}u_{(i_1)}
+
\sum_{(i_1,i_2)}a_{(i_1,i_2)}u_{(i_1,i_2)}
+\cdots+
\sum_{(i_1,i_2,\ldots,i_{s-1})}
a_{(i_1,i_2,\ldots,i_{s-1})}u_{(i_1,i_2,\ldots,i_{s-1})},
\end{multline*}
in which the sum $\ds\sum_{(i_1,i_2,\ldots,i_{k})}$ is taken over all
the distinct monomials $u_{(i_1,i_2,\ldots,i_{k})}=u_{i_1}\cdots u_{i_k}$
of degree $k$ and $a_{(i_1,i_2,\ldots,i_{k})}\in R$. Then $\epsilon^{(s)}$
is surjective and in $\mathbf{K}(R;I^s)_0$ we have
\[
\im\d^{(s)}\subset\ker\epsilon^{(s)}.
\]
On the other hand, suppose that
\[
\mathbf{a}=(a_0,\tilde a_1,\ldots,\tilde a_{s-1})\in\ker\epsilon^{(s)},
\]
where
\[
\tilde a_k=
\sum_{(i_1,\ldots,i_k)}a_{(i_1,\ldots,i_k)}\tilde u_{(i_1,\ldots,i_k)}.
\]
Then writing
\[
a_k=\sum_{(i_1,\ldots,i_k)}a_{(i_1,\ldots,i_k)}u_{(i_1,\ldots,i_k)},
\]
we find
\[
a_0+a_1+\cdots+a_{s-1}\in I^s,
\]
so $a_0\in I$. This means that
\[
\mathbf{a}\equiv(0,\tilde b_1,\tilde a_2\ldots,\tilde a_{s-1})\mod\im\d^{(s)}.
\]
Repeating this argument modulo higher powers of $I$, we find that
\begin{align*}
\mathbf{a}&\equiv(0,0,\ldots,0,\tilde b_{s-1})\mod\im\d^{(s)}, \\
\intertext{where}
\tilde b_{s-1}&=
\sum_{(i_1,\ldots,i_{s-1})}
a_{(i_1,\ldots,i_{s-1})}\tilde u_{(i_1,\ldots,i_{s-1})} \\
\intertext{and}
b_{s-1}&=
\sum_{(i_1,\ldots,i_{s-1})}
a_{(i_1,\ldots,i_{s-1})}u_{(i_1,\ldots,i_{s-1})}\in I^s.
\end{align*}
But taking
\[
c=
\sum_{(i_1,\ldots,i_{s-1})}
a_{(i_1,\ldots,i_{s-1})}\tau_{i_{s-1}}\tilde
u_{(i_1,\ldots,i_{s-2})},
\]
we find
\[
\d^{(s)}(0,\ldots,0,c)=(0,0,\ldots,0,\tilde b_{s-1}).
\]
Hence $\mathbf{a}\in\im\d^{(s)}$. This shows that
\[
\ker\epsilon^{(s)}=\im\d^{(s)}.
\]

Suppose that $n\geq1$ and
\[
\mathbf{x}=(x_0,x_1,\ldots,x_{s-1})\in\mathbf{K}(R;I^s)_n
\]
satisfies $\d^{(s)}\mathbf{x}=0$. Then $x'_0=0$ and so by exactness
of $\mathbf{Q}^{(0)}_*$,
\[
x_0=\d_{\mathbf{Q}}^{(0)}y_0
\]
for some $y_0\in\mathbf{Q}^{(0)}_{n+1}$. Then
\begin{align*}
0&=x'_1=\del^{(1)}\d_{\mathbf{Q}}^{(0)}y_0+\d_{\mathbf{Q}}^{(1)}x_1 \\
&=\d_{\mathbf{Q}}^{(1)}(-\del^{(1)}y_0+x_1),
\end{align*}
hence by exactness of $\mathbf{Q}^{(1)}_*$,
\[
x_1=\d_{\mathbf{Q}}^{(1)}y_1+\del^{(1)}y_0
\]
for some $y_1\in\mathbf{Q}^{(1)}_{n+1}$. Continuing in this way,
eventually we obtain an element
\[
(y_0,y_1,\ldots,y_{s-1})\in\mathbf{K}(R;I^s)_{n+1}
\]
for which
\[
x_k=\d_{\mathbf{Q}}^{(k)}y_k+\del^{(k)}y_{k-1}\quad(1\leq k\leq s-1).
\]

\begin{thm}\label{thm:R/Is-ProjRes}
For $s\geq1$,
\[
\mathbf{K}(R;I^s)_*\xrightarrow{\epsilon^{(s)}}R/I^s\ra0
\]
is a resolution by free $R$-modules.
\end{thm}

The complex $(\mathbf{K}(R;I^s)_*,\d^{(s)})$ has a multiplicative
structure coming from the pairings
\[
\mathbf{Q}^{(p)}_*\otimes\mathbf{Q}^{(q)}_*\lra\mathbf{Q}^{(p+q)}_*;
\quad
(x\tilde u_{(i_1,\ldots,i_p)})\otimes(y\tilde u_{(j_1,\ldots,j_q)})
\longmapsto
(xy)\tilde u_{(i_1,\ldots,i_p,j_1,\ldots,j_q)}.
\]
\begin{thm}\label{thm:R/Is-ProjResDGA}
For $s\geq1$, the complex $(\mathbf{K}(R;I^s)_*,\d^{(s)})$ is
a differential graded $R$-algebra, providing a multiplicative
resolution free resolution of $R/I^s$ over $R$.
\end{thm}
\begin{cor}\label{cor:R/Is-ProjResDGA}
As an $R/I$-algebra,
\[
\Tor^R_*(R/I,R/I^s)=
\mathrm{H}_*(R/I\otimes\mathbf{K}(R;I^s)_*,1\otimes\d^{(s)}).
\]
\end{cor}
Notice that in the differential graded $R/I$-algebra
$(R/I\otimes\mathbf{K}(R;I^s)_*,1\otimes\d^{(s)})$ we have
\begin{equation}\label{eqn:R/ItensorK(R,Is)-d}
1\otimes\d^{(s)}(t\otimes(x_0,x_1,\ldots,x_{s-1}))=
t\otimes(0,\del^{(1)}x_0,\del^{(2)}x_1,\ldots,\del^{(s-2)}x_{s-2}).
\end{equation}
We will exploit this in the next section.

\section{A spectral sequence}\label{sec:SS}

In order to compute $\Tor^R_*(R/I,R/I^s)$ explicitly we will set
up a double complex and consider one of the two associated spectral
sequences~\cite{Weibel}. We begin by defining the double complex
$(\mathrm{P}_{*,*},\dh,\dv)$ with
\begin{align*}
\mathrm{P}_{p,q}&=\mathbf{Q}^{(p)}[-p]_{q+p}
(\text{$=\mathbf{Q}^{(p)}_q$ as $R$-modules}), \\
\dh&=(-1)^p\del^{(p+1)}[-p]=\del^{(p+1)}, \\
\dv&=(-1)^p\d_{\mathbf{Q}}^{(p)}[-p]=\d_{\mathbf{Q}}^{(p)}.
\end{align*}
Considered as a homomorphism
\[
\dv\dh+\dh\dv\:\mathrm{P}_{p,q}\lra\mathrm{P}_{p+1,q+1},
\]
we have from Equation~\eqref{eqn:del-dQ},
\[
\dv\dh+\dh\dv=
\d_{\mathbf{Q}}^{(p+1)}\del^{(p+1)}+\del^{(p+1)}\d_{\mathbf{Q}}^{(p)}
=0.
\]
As the associated (direct sum) total complex $(\Tot\mathrm{P}_*,\dt)$
we obtain
\[
\Tot\mathrm{P}_n=\bigoplus_{k}\mathrm{P}_{k,n-k},
\quad
\dt=\dh+\dv.
\]
Notice that
\[
\Tot\mathrm{P}_n=\mathbf{K}(R;I^s)_n,
\quad
\dt=\d^{(s)}
\]
Hence
\[
\mathrm{H}_*(\Tot\mathrm{P}_*,\dt)=R/I^s.
\]
Applying the functor $R/I\otimes(\ )$ we obtain another double
complex $(\bar{\mathrm{P}}_{*,*},\dh,\dv)$ where
\[
\bar{\mathrm{P}}_{p,q}=R/I\otimes\mathrm{P}_{p,q}.
\]
The associated total complex $(\Tot\bar{\mathrm{P}}_*,\dt)$ has
\[
\Tot\bar{\mathrm{P}}_n=R/I\otimes\mathbf{K}(R;I^s)_n,
\quad
\dt=1\otimes\d^{(s)}
\]
and homology
\[
\mathrm{H}_*(\Tot\bar{\mathrm{P}}_*,\dt)=\Tor^R_*(R/I,R/I^s).
\]

Filtering by columns we obtain a spectral sequence with
\begin{equation}\label{eqn:SS}
\E^2_{p,q}=
\mathrm{H}_p(\mathrm{H}_q(\bar{\mathrm{P}}_{*,*},\dv),\dh)
\Lra\Tor^R_{p+q}(R/I,R/I^s).
\end{equation}
Here
\[
\mathrm{H}_*(\bar{\mathrm{P}}_{p,*},\dv)=
\mathrm{H}_*(R/I\otimes\mathbf{Q}^{(p)}_*,1\otimes\d_{\mathbf{Q}}^{(p)})
=\Tor^R_*(R/I,I^p/I^{p+1})
\]
and $\mathrm{H}_*(\mathrm{H}_q(\bar{\mathrm{P}}_{*,*},\dv),\dh)$
is the homology of the complex
\begin{multline*}
0\ra
\Tor^R_q(R/I,R/I)\xrightarrow{\del^{(1)}_*}\Tor^R_{q-1}(R/I,I/I^2)\lra \\
\cdots\lra\Tor^R_{q-s+2}(R/I,I^{s-2}/I^{s-1})\xrightarrow{\del^{(s-1)}_*}
\Tor^R_{q-s+1}(R/I,I^{s-1}/I^s)\ra0.
\end{multline*}
\begin{lem}\label{lem:R/I-delexact}
For $s\geq2$, the complex of graded $R/I$-modules
\begin{multline*}
\Tor^R_*(R/I,R/I)\xrightarrow{\del^{(1)}_*}\Tor^R_*(R/I,I/I^2)\lra \\
\cdots\lra\Tor^R_*(R/I,I^{s-2}/I^{s-1})\xrightarrow{\del^{(s-1)}_*}
\Tor^R_*(R/I,I^{s-1}/I^s)
\end{multline*}
is exact, hence the spectral sequence of~\emph{\eqref{eqn:SS}}
collapses at $\E^2$ to give
\[
\Tor^{R}_n(R/I,R/I^s)=
\begin{cases}
R/I&\text{\rm if $n=0$}, \\
\coker\del^{(s-1)}_*\:\Tor^R_n(R/I,I^{s-2}/I^{s-1})\lra\Tor^R_n(R/I,I^{s-1}/I^s)
&\text{\rm if $n\neq0$}.
\end{cases}
\]
With its natural $R/I$-algebra structure, $\Tor^{R}_*(R/I,R/I^s)$ has
trivial products.
\end{lem}
\begin{proof}
Our proof uses the observation that this complex is equivalent to part of the
Koszul complex $\Lambda_{R/I[\tilde u_i:i]}(\tilde e_i:i)$ which provides a
free resolution of $R/I=R/I[\tilde u_i:i]/(\tilde u_i:i)$ as an
$R/I[\tilde u_i:i]$-module. Up to a sign, the differential $\tilde\d$ agrees
with that of the complex in Lemma~\ref{lem:R/I-delexact}. The result follows
by exactness of the Koszul complex since the generators
$\tilde u_1,\tilde u_2,\ldots$ form a regular sequence in $R/I[\tilde u_i:i]$.
We now proceed to give the details.

For a commutative unital ring $\k$, make $\Lambda_{\k[\tilde u_i:i]}(\tilde e_i:i)$
a bigraded $\k$-algebra for which
\[
\bideg\tilde e_i=(1,0),\quad\bideg\tilde u_i=(1,-1).
\]
For each grading $p\geq0$ of $\Lambda_{\k[\tilde u_i:i]}(\tilde e_i:i)$,
\[
\Lambda_{\k[\tilde u_i:i]}(\tilde e_i:i)^p=
\bigoplus_{q\geq0}\Lambda_{\k[\tilde u_i:i]}(\tilde e_i:i)^{p+q,-q}
\]
and the differential
\[
\d^p\:\Lambda_{\k[\tilde u_i:i]}(\tilde e_i:i)^p
\lra
\Lambda_{\k[\tilde u_i:i]}(\tilde e_i:i)^{p+1}
\]
decomposes as a sum of components
\[
\d^{p+q,-q}\:\Lambda_{\k[\tilde u_i:i]}(\tilde e_i:i)^{p+q,-q}
\lra
\Lambda_{\k[\tilde u_i:i]}(\tilde e_i:i)^{p+q,-q-1},
\]
since
\[
\d^p(\tilde e_{i_1}\cdots\tilde e_{i_p}\tilde u_{j_1}\cdots\tilde u_{j_q})
=
\sum_{k=1}^p(-1)^{k-1}
\tilde e_{i_1}\cdots\tilde e_{i_{k-1}}\tilde e_{i_{k+1}}\cdots\tilde e_{i_p}
\tilde u_{i_k}\tilde u_{j_1}\cdots\tilde u_{j_q}.
\]
Exactness of $\d$ on $\Lambda_{\k[\tilde u_i:i]}(\tilde e_i:i)$ is equivalent
to the fact that for all pairs $p,q$,
\[
\ker\d^{p+q,-q}=\im\ker\d^{p+q,-q+1}.
\]
Hence for all $q$ we have
\[
\bigoplus_{p\geq0}\ker\d^{p+q,-q}=\bigoplus_{p\geq0}\im\d^{p+q,-q+1},
\]
which is equivalent to the exactness of
\[
\Lambda_{\k[\tilde u_i:i]}(\tilde e_i:i)\oTimes{\k}\k[\tilde u_i:i]_{q-1}
\xrightarrow{\d}
\Lambda_{\k[\tilde u_i:i]}(\tilde e_i:i)\oTimes{\k}\k[\tilde u_i:i]_q
\xrightarrow{\d}
\Lambda_{\k[\tilde u_i:i]}(\tilde e_i:i)\oTimes{\k}\k[\tilde u_i:i]_{q+1},
\]
where $\k[\tilde u_i:i]_n\subset\k[\tilde u_i:i]$ denotes the homogeneous
polynomials of degree $n$.

The statement about products is now immediate since the spectral sequence is
clearly multiplicative. Actually the full force of this is not really needed
since
\[
\Tor^{R}_*(R/I,R/I^s)\iso R/I\oplus\coker\del^{(s-1)}_*
\]
and products of elements in the bottom filtration $\coker\del^{(s-1)}_*$ are
zero in $\E^\infty=\E^2$.
\end{proof}
We can strengthen our hold on $\Tor^R_*(R/I,R/I^s)$ using the ideas in the
last proof.
\begin{prop}\label{prop:Tor-freeR/I}
For $s\geq1$, $\Tor^R_*(R/I,R/I^s)$ is a free $R/I$-module.
\end{prop}
\begin{proof}
The case $s=1$ is of course a consequence of Corollary~\ref{cor:KoszulRes}.

Using the notation of the proof of Lemma~\ref{lem:R/I-delexact}, notice that
in terms of the $\k$-basis of elements
$\tilde e_{i_1}\cdots\tilde e_{i_p}\tilde u_{j_1}\cdots\tilde u_{j_q}$, each
$\d^{p+q,-q}$ is actually given by a $\Z$-linear combination.
Therefore we can reduce to the case where $\k=\Z$, and then tensor up over $\Z$
with an arbitrary $\k$.

For each pair $p,q\geq0$, $\Lambda_{\Z[\tilde u_i:i]}(\tilde e_i:i)^{p+q,-q}$
breaks up into a direct sum of $\Z$-submodules $M^{p+q,-q}(S)$ where $S$ is a
set of exactly $p+q$ elements of the indexing set for the $u_i$'s and $M^{p+q,-q}(S)$
is spanned by the finitely many elements
$\tilde e_{i_1}\cdots\tilde e_{i_p}\tilde u_{j_1}\cdots\tilde u_{j_q}$
with
\[
S=\{i_1,\ldots,i_p,j_1,\ldots j_q\},
\quad
i_1<i_2<\cdots<i_p.
\]
Notice that on restriction we have
\[
\d^{p+q,-q}_{M(S)}=\d^{p+q,-q}\:M^{p+q,-q}(S)\lra M^{p+q,-q-1}(S).
\]
By exactness, $\im\d^{p+q,-q+1}_{M(S)}=\ker\d^{p+q,-q}_{M(S)}$. Since $M^{p+q,-q}(S)$
is a finitely generated free module, $\ker\d^{p+q,-q}_{M(S)}$ is indivisible in
$M^{p+q,-q}(S)$ and so is a summand. Hence $\im\d^{p+q,-q+1}_{M(S)}$ is always a
summand of $M^{p+q,-q}(S)$. Taking the sum over all $S$ and then over all $p$ we
find that for each $q$,
\[
\im\d\:\Lambda_{\Z[\tilde u_i:i]}(\tilde e_i:i)\oTimes{\Z}\Z[\tilde u_i:i]_{q-1}
\lra
\Lambda_{\Z[\tilde u_i:i]}(\tilde e_i:i)\oTimes{\Z}\Z[\tilde u_i:i]_q
\]
is a summand in
\[
\Lambda_{\Z[\tilde u_i:i]}(\tilde e_i:i)\oTimes{\Z}\Z[\tilde u_i:i]_q.
\qedhere
\]
\end{proof}

\section{Appendix: Resolutions of extensions}\label{sec:ResExtns}

In this Appendix we recall some standard facts about extensions of $R$-modules,
see~\cite{Weibel}, and also give an interpretation in terms of the derived
category of complexes of $R$-modules. Our aim is to put the construction of
the complex $(\mathbf{K}(R;I^s)_*,\d^{(s)})$ into a broader context for the
benefit of those unfamiliar with such ideas. In fact, we found this complex
by iterating the splicing construction for the resolution of an extension
given below; in our case this works well to give a very concrete and manageable
resolution.

Suppose that
\begin{equation}\label{eqn:Ext}
\mathcal{E}\:\quad 0\ra L\lra M\lra N\ra0
\end{equation}
is a short exact sequence of $R$-modules and
\[
P_*\xrightarrow{\epsilon}N\ra0
\]
is a projective resolution of $N$. Then there are homomorphisms
$\epsilon_0\:P_0\lra M$ and $\epsilon_1\:P_1\lra L$ which fit into
a commutative diagram
\begin{equation}\label{eqn:Ext-Res}
\begin{CD}
0@<<< N    @<<< P_0@<<< P_1@<<< P_2@<<<\cdots \\%
@.    @|   @V\epsilon_0VV    @V\epsilon_1VV     \\
0@<<< N    @<<< M @<<<   L  @<<< 0 @. @.
\end{CD}
\end{equation}
Then $\epsilon_1$ is a cocycle in $\Hom_R(P_1,L)$ which represents
an element $\Theta(\mathcal{E})\in\Ext^1_R(N,L)$ classifying the
extension $\mathcal{E}$.

Now let $Q_*\xrightarrow{\eta}L\ra0$ be a projective resolution of
$L$ with differential $d_Q$ and $Q[-1]_*$ its suspension. Then
the differential $d_{Q}[-1]$ in $Q[-1]_*$ is given by
\[
d_{Q}[-1]x=-d_Qx.
\]
It is well known that in the derived category $\mathcal{D}^\flat(R)$
of bounded below complexes of $R$-modules,
\begin{equation}\label{eqn:Ext1=HomDerivCat[-1]}
\Ext^1_R(N,L)\iso\Hom_{\mathcal{D}^\flat(R)}(P_*,Q[-1]_*).
\end{equation}
Given the diagram~\eqref{eqn:Ext-Res}, there is an extension to
a diagram
\begin{equation}\label{eqn:Ext-Res-extension}
\begin{CD}
0@<<< N    @<<< P_0@<<< P_1@<<< P_2@<<<     P_3@<<<\cdots \\%
@.    @|   @V\epsilon_0VV  @V\epsilon'_1VV     @V\epsilon'_2VV @V\epsilon'_3VV   \\
0@<<< N    @<<< M @<<< Q_0 @<<<    Q_1         @<<<Q_2@<<< \cdots
\end{CD}
\end{equation}
and hence the element
\[
\begin{CD}
0 @<<< P_0@<<< P_1@<<< P_2@<<<     P_3@<<<\cdots \\%
    @|   @V0VV    @V\epsilon'_1VV @V\epsilon'_2VV  @V\epsilon'_3VV   \\
0 @<<< 0 @<<<  Q_0 @<<<    Q_1 @<<<Q_2@<<< \cdots
\end{CD}
\]
which represents an element of $\Hom_{\mathcal{D}^\flat(R)}(P_*,Q[-1]_*)$.
Conversely, a diagram with exact rows such as~\eqref{eqn:Ext-Res-extension}
clearly gives rise to an extension of the form~\eqref{eqn:Ext-Res}. Perhaps
a more illuminating way to view this morphism in $\mathcal{D}^\flat(R)$ is
in terms of the diagram
\[
\begin{CD}
0 @<<< P_0@<<< P_1@<<< P_2@<<<     P_3@<<<\cdots \\
@|   @V0VV    @V\epsilon_1VV @VVV  @VVV   \\
0 @<<< 0 @<<<  L[-1] @<<<    0 @<<<0@<<< \cdots             \\
@|   @AAA    @A\epsilon AA @AAA  @AAA  \\
0 @<<< 0 @<<<  Q[-1]_0 @<<<    Q[-1]_1 @<<<Q[-1]_2@<<< \cdots
\end{CD}
\]
where the augmentation $\epsilon\:Q[-1]_*\lra L[-1]$ is a homology
equivalence, hence an isomorphism in $\mathcal{D}^\flat(R)$, so
the composite
\[
P_*\xrightarrow{\ph{a}\epsilon_1\ph{b}}L\xrightarrow{\epsilon^{-1}}Q_*
\]
gives an element of $\Hom_{\mathcal{D}^\flat(R)}(P_*,Q[-1]_*)$.
Of course all of these classes agree with $\Theta(\mathcal{E})$.
Notice that $\Theta(\mathcal{E})$ is determined by the homomorphism
$\epsilon_1\:P_1\lra Q[-1]_1=Q_0$ lifting the map $P_1\lra L$.

We also recall a well known related result, see~\cite{Weibel}.
\begin{prop}\label{prop:SES-SplicingRes}
For a ring $R$, let
\[
0\la A\lla B\lla C\la0
\]
be short exact and $P_*\lra A\ra0$ and $Q_*\lra C\ra0$ projective
resolutions. Then there is a projective resolution of the
form $(P\oplus Q)_*\lra B\ra0$ and a commutative diagram
\[
\begin{CD}
0@<<<P_*@<<<(P\oplus Q)_*@<<<Q_*@<<<0  \\
@. @VVV      @VVV    @VVV @. \\
0@<<<A@<<<B@<<<C@<<<0  \\
\end{CD}
\]
\end{prop}
\begin{proof}
The extension is classified by an element of
$\Hom_{\mathcal{D}^\flat(R)}(P_*,Q[-1]_*)$ corresponding
to a chain map $\del_*\:P_*\lra Q[-1]_*$. Viewed as a
sequence of maps $\del_n\:P_n\lra Q[-1]_{n-1}$, $\del_*$
must satisfy
\begin{equation}\label{eqn:Ext-DiffsCompatable}
\d_Q\del_n+\del_{n-1}\d_P=0\quad(n\geq1).
\end{equation}
The formula
\[
\d(x,y)=(\d x,\del_nx+\d_Qy)\quad(x\in P_n,\;y\in Q_n)
\]
defines the differential in $(P\oplus Q)_*$.
\end{proof}


\begin{thebibliography}{22}\frenchspacing
\bibitem{AB+AJ:Brave-BocksteinOps}
A. Baker \& A.~Jeanneret,
Brave new Bockstein operations,
Glasgow University Mathematics Department preprint 03/18.
\bibitem{AB+AL:Rmod-ASS}
A. Baker \& A. Lazarev,
On the Adams Spectral Sequence for $R$-modules,
Algebraic \& Geometric Topology \textbf{1} (2001), 173--99.
\bibitem{CE}
H.~Cartan \& S.~Eilenberg,
Homological Algebra,
Princeton University Press (1956).
\bibitem{Matsumura}
H.~Matsumura,
Commutative Ring Theory,
Cambridge University Press (1986).
\bibitem{Smith:EM}
L.~Smith,
Homological algebra and the Eilenberg-Moore spectral sequence,
Trans. Amer. Math. Soc. \textbf{129} (1967), 58--93.
\bibitem{Strickland:MU}
N.~P.~Strickland,
Products on $MU$-modules,
Trans. Amer. Math. Soc. \textbf{351} (1999), 2569--2606.
\bibitem{Tate}
J. Tate,
Homology of Noetherian rings and of local algebras,
Illinois J. Math. \textbf{1} (1957), 14--27.
\bibitem{Weibel}
C.~A.~Weibel,
An Introduction to Homological Algebra,
Cambridge University Press (1994).
\bibitem{SW:HomPowRegId}
S. W\"uthrich,
$I$-adic towers in topology,
Algebr. Geom. Topol. \textbf{5} (2005), 1589–-1635
\end{thebibliography}
\end{document}